\documentclass[reqno,oneside]{amsart}

\usepackage[english]{babel} 
\usepackage[latin1]{inputenc} 
\usepackage[T1]{fontenc} 
\usepackage{lmodern}
\usepackage{amsmath,amssymb,amsfonts,amsthm,mathrsfs,latexsym} 
\usepackage{graphicx} 
\usepackage{verbatim} 
\usepackage[all]{xy} 
\usepackage{nicefrac}
\usepackage[hyperfootnotes=false, plainpages=false, pdfpagelabels, colorlinks, linkcolor=blue, citecolor=blue, urlcolor=blue, hypertexnames=true]{hyperref} 
\usepackage{tikz}
\usetikzlibrary{patterns}
\usepackage{varioref}
\usepackage{multirow}

\makeatletter
\renewenvironment{proof}[1][\proofname]{\par
  \vspace{-\topsep}
  \pushQED{\qed}%
  \normalfont
  \topsep-6pt \partopsep0pt 
  \trivlist
  \item[\hskip\labelsep
        \itshape
    #1\@addpunct{.}]\ignorespaces
}{%
  \popQED\endtrivlist\@endpefalse
  \addvspace{6pt plus 6pt} 
}
\makeatother

\newtheorem{thm}{Theorem}
\newtheorem{lem}[thm]{Lemma}
\newtheorem{prop}[thm]{Proposition}
\newtheorem{cor}[thm]{Corollary}
\newtheorem{que}[thm]{Question}
\newtheorem*{thm*}{Theorem}
\newtheorem*{claim*}{Claim}

\theoremstyle{definition}

\newcommand{\mc}{\mathcal}
\newcommand{\mb}{\mathbb}
\newcommand{\mr}{\mathrm}

\newcommand{\C}{\mathbb{C}}
\newcommand{\R}{\mathbb{R}}
\newcommand{\Z}{\mathbb{Z}}

\newcommand{\la}{\langle}
\newcommand{\ra}{\rangle}
\renewcommand{\epsilon}{\varepsilon}
\renewcommand{\phi}{\varphi}
\renewcommand{\tilde}{\widetilde}
\renewcommand{\hat}{\widehat}

\renewcommand{\Im}{\mr{Im}}
\newcommand{\ImF}{\Im\;\mb F}

\newcommand{\acts}{\curvearrowright}
\newcommand{\Cs}{\ensuremath{C^*}}

\newcommand{\FF}{\mr{F}_{4(-20)}}

\DeclareMathOperator{\SL}{SL}
\DeclareMathOperator{\SO}{SO}
\DeclareMathOperator{\SU}{SU}
\DeclareMathOperator{\Sp}{Sp}

\DeclareMathOperator{\PSL}{PSL}

\DeclareMathOperator{\Ind}{Ind}

\numberwithin{equation}{section}
\begin{document}
\selectlanguage{english} 

\begin{abstract}
We study the following question: given a locally compact group when does its Fourier algebra coincide with the subalgebra of the Fourier-Stieltjes algebra consisting of functions vanishing at infinity? We provide sufficient conditions for this to be the case.

As an application, we show that when $P$ is the minimal parabolic subgroup in one of the classical simple Lie groups of real rank one or the exceptional such group, then the Fourier algebra of $P$ coincides with the subalgebra of the Fourier-Stieltjes algebra of $P$ consisting of functions vanishing at infinity. In particular, the regular representation of $P$ decomposes as a direct sum of irreducible representations although $P$ is not compact.
\end{abstract}


\title{Fourier Algebras of Parabolic Subgroups
}
\author{S{\o}ren Knudby}
\thanks{Supported by ERC Advanced Grant no.~OAFPG 247321 and the Danish National Research Foundation through the Centre for Symmetry and Deformation (DNRF92).}
\address{Department of Mathematical Sciences, University of Copenhagen,
\newline Universitetsparken 5, DK-2100 Copenhagen \O, Denmark}
\email{knudby@math.ku.dk}
\date{\today}
\maketitle
\section{Introduction}
In the paper \cite{MR0228628}, Eymard introduced the Fourier algebra $A(G)$ and the Fourier-Stieltjes algebra $B(G)$ of a locally compact group $G$. The Fourier-Stieltjes algebra $B(G)$ is defined as the linear span of the continuous positive definite functions on $G$. There is a natural identification of $B(G)$ with the Banach space dual of the full group \Cs-algebra $C^*(G)$, and under this identification $B(G)$ inherits a norm with which it is a Banach space. The Fourier algebra $A(G)$ is the closed subspace in $B(G)$ generated by the compactly supported functions in $B(G)$. Other descriptions of $A(G)$ and $B(G)$ are available (see Section~\ref{sec:fourier}). The Fourier and Fourier-Stieltjes algebras play an important role in non-commutative harmonic analysis.

For any locally compact group it is the case that elements of the Fourier algebra vanish at infinity: $A(G) \subseteq C_0(G)$. It is natural to ask whether the converse is true, that is, if every function in $B(G)$ vanishing at infinity belongs to $A(G)$.

\begin{que}\label{q1}
Let $G$ be a locally compact group. When does the equality
\begin{align}\label{eq:q1}
A(G) = B(G) \cap C_0(G)
\end{align}
hold?
\end{que}
Of course, if the group $G$ is compact then $B(G) = A(G)$, and \eqref{eq:q1} obviously holds. But for non-compact groups the question is more delicate.

In 1916, Menchoff \cite{menchoff} proved the existence of a singular probability measure $\mu$ on the circle such that its Fourier-Stieltjes transform $\hat\mu$ satisfies $\hat\mu(n) \to 0$ as $|n|\to\infty$. In other words, $\hat\mu \in B(\Z) \cap C_0(\Z)$, but $\hat\mu\notin A(\Z)$, and thus the answer to Question~\ref{q1} is negative when $G$ is the group of integers $\Z$. In 1966, Hewitt and Zuckerman \cite{MR0193435} proved that for any abelian locally compact group $G$ the answer to Question~\ref{q1} is always negative, unless $G$ is compact. In 1983, Taylor showed that for any countable, discrete group $G$ one has $A(G) \neq B(G) \cap C_0(G)$, unless $G$ is finite (see \cite[p.~190]{MR690194} and \cite{MR552704}). In fact, Taylor proved that non-compact, second countable IN-groups $G$ never satisfy \eqref{eq:q1}.

It is proved in \cite{MR0493175},\cite{MR552704} that if \eqref{eq:q1} holds for some second countable, locally compact group $G$, then the regular representation of $G$ is completely reducible, i.e., a direct sum of irreducible representations. For a while, this was thought to be a characterization of groups satisfying \eqref{eq:q1}, but this was shown not to be the case (see \cite{MR509261}).

The first non-compact example of a group $G$ satisfying \eqref{eq:q1} was given by Khalil in \cite{MR0350330}, namely the (non-unimodular) $ax+b$ group consisting of affine transformations $x\mapsto ax + b$ of the real line, where $a> 0$ and $b\in\R$. We remark that the $ax+b$ group is isomorphic to the minimal parabolic subgroup in the simple Lie group $\PSL_2(\R)$ of real rank one.

It follows from Baggett's work \cite{MR735532} that if $G$ is a locally compact, second countable group which is also connected, unimodular and has a completely reducible regular representation, then $G$ is compact (see \cite[Theorem~3]{MR2459312}). In particular, Question~\ref{q1} has a negative answer for locally compact, second countable, connected, unimodular groups which are non-compact. This gives an abundance of examples of groups where Question~\ref{q1} has a negative answer. An example given in \cite{MR0486289} and \cite{MR0420149} (independently) of a unimodular group satisfying \eqref{eq:q1} shows that the assumption about connectedness cannot be removed from the previous statement, and of course the assumption about unimodularity cannot be removed as the $ax +b$ group shows.

It should be apparent from the above that there are plenty of examples of groups for which Question~\ref{q1} has a negative answer and so far only very few examples with an affirmative answer. In this paper we provide new examples of groups answering Question~\ref{q1} in the affirmative. Our main source of examples is formed by the minimal parabolic subgroups in connected simple Lie groups of real rank one. But first we give a more straightforward example which is a subgroup of $\SL_3(\R)$ resembling the $ax+b$ group. We prove the following.

\begin{thm}\label{thm:fourier-PP}
We have $A(P) = B(P) \cap C_0(P)$ for the group
\begin{align}\label{eq:PP}
P = 
\left\{\begin{pmatrix}
\lambda & a & c \\
0 & \lambda^{-1} & b \\
0 & 0 & 1
\end{pmatrix}
\middle\vert a,b,c\in\R,\ \lambda>0
\right\}.
\end{align}
\end{thm}

If we think of $\SL_2(\R)\ltimes\R^2$ as a subgroup of $\SL_3(\R)$ in the following way
\begin{align}\label{eq:SL3}
\left(\begin{array}{cc|c}
\multicolumn{2}{c|}{\multirow{2}{*}{$\SL_2(\R)$}} & \multirow{2}{*}{$\R^2$} \\
 & & \\
 \hline
\multicolumn{2}{c|}{0} & 1
\end{array}\right),
\end{align}
then we can think of $P$ as a subgroup of $\SL_2(\R)\ltimes\R^2$. This viewpoint is relevant in \cite{HK-WH-examples} and was actually the initial motivation for the present paper.

Apart from the group in \eqref{eq:PP}, our examples of groups satisfying \eqref{eq:q1} arise in the following way. Let $n\geq 2$, let $G$ be one of the classical simple Lie groups $\SO_0(n,1)$, $\SU(n,1)$, $\Sp(n,1)$ or the exceptional group $\FF$, and let $G = KAN$ be the Iwasawa decomposition. If $M$ is the centralizer of $A$ in $K$, then $P = MAN$ is the \emph{minimal parabolic subgroup} of $G$. We refer to Section~\ref{sec:simple-groups} for more details on these groups. We prove the following theorem concerning the Fourier algebras of minimal parabolic subgroups.

\begin{thm}\label{thm:fourier-P}
Let $P$ be the minimal parabolic subgroup in one of the simple Lie groups $\SO_0(n,1)$, $\SU(n,1)$, $\Sp(n,1)$ or $\FF$. Then $A(P) = B(P) \cap C_0(P)$.
\end{thm}

In order to establish Theorem~\ref{thm:fourier-PP} and Theorem~\ref{thm:fourier-P} we develop a general strategy for providing examples of groups that answer Question~\ref{q1} affirmatively. The strategy is based on (1) determining all irreducible representations of the group, (2) determining the irreducible subrepresentations of the regular representation and (3) disintegration theory. An often useful tool for (1) is the Mackey Machine (see \cite[Chapter 6]{MR1397028} and \cite{MR3012851}). Our strategy is contained in the following theorem.

\begin{thm}\label{thm:strategy}
Let $G$ be a second countable, locally compact group satisfying the following two conditions.
\begin{enumerate}
	\item $G$ is of type I.
	\item There is a non-compact, closed subgroup $H$ of $G$ such that every irreducible unitary representation of $G$ is either trivial on $H$ or is a subrepresentation of the left regular representation $\lambda_G$.
\end{enumerate}
Then
$$
A(G) = B(G) \cap C_0(G).
$$
In particular, the left regular representation $\lambda_G$ is completely reducible.
\end{thm}

In order to verify the two conditions in Theorem~\ref{thm:strategy} for the minimal parabolic subgroups $P$, we rely primarily on earlier work of J.A.~Wolf. In \cite{MR0422519} the irreducible representations of some parabolic subgroups are determined by employing the Mackey Machine, and the approach of \cite{MR0422519} carries over to our situation almost without changes. Combining \cite{MR0422519} with \cite{MR0342641} we can easily determine the irreducible subrepresentations of the regular representation of $P$.

Although the algebra $B(G)\cap C_0(G)$ often does not coincide with the Fourier algebra, it has gained much interest recently (see \cite{MR3473390} and \cite{MR3071703}). It is sometimes referred to as the Rajchman algebra.

The organization of the paper is as follows. Section~\ref{sec:fourier} contains the basic properties of the Fourier and Fourier-Stieltjes algebra. In Section~\ref{sec:strategy} we prove Theorem~\ref{thm:strategy}. Section~\ref{sec:inv-meas} contains a few results to be used later when we verify condition~(2) of Theorem~\ref{thm:strategy} for the groups under consideration in the succeeding sections. In Section~\ref{sec:PP} we prove Theorem~\ref{thm:fourier-PP}, and in Section~\ref{sec:simple-groups} we prove Theorem~\ref{thm:fourier-P}. Finally, Section~\ref{sec:concluding} contains some concluding remarks.

\section{The Fourier and Fourier-Stieltjes algebra}\label{sec:fourier}

This section contains a brief description of the Fourier and Fourier-Stieltjes algebra of a locally compact group introduced by Eymard in \cite{MR0228628}. We refer to the original paper \cite{MR0228628} for more details. Let $G$ be a locally compact group equipped with a left Haar measure. By a representation of $G$ we always mean a strongly continuous unitary representation of $G$ on some Hilbert space. If $\pi$ is a representation of $G$ on a Hilbert space $\mc H$ and $x,y\in\mc H$, then the continuous complex function
$$
\phi(g) = \la \pi(g)x,y\ra, \qquad g\in G,
$$
is a \emph{matrix coefficient} of $\pi$. The Fourier-Stieltjes algebra of $G$ is denoted $B(G)$ and consists of the complex linear span of continuous positive definite functions on $G$. It coincides with the set of all matrix coefficients of representations of $G$,
$$
B(G) = \{ \la \pi(\cdot)x,y \ra \mid (\pi,\mc H) \text{ is a representation of } G \text{ and } x,y\in\mc H \}.
$$
Since the pointwise product of two positive definite functions is again positive definite, $B(G)$ is an algebra under pointwise multiplication. Given $\phi\in B(G)$, the map
$$
f\mapsto \la f,\phi\ra = \int_G f(x)\phi(x) \,dx, \qquad f\in L^1(G),
$$
is a linear functional on $L^1(G)$ which is bounded, when $L^1(G)$ is equipped with the universal \Cs-norm. Hence $\phi$ defines a functional on $C^*(G)$, the full group \Cs-algebra of $G$, and this gives the identification of $B(G)$ with $C^*(G)^*$ as vector spaces. The Fourier-Stieltjes algebra inherits the norm
$$
\|\phi\| = \sup\{ |\la f,\phi\ra| \mid f\in L^1(G),\ \|f\|_{C^*(G)} \leq 1 \}
$$
of $C^*(G)^*$ from this identification. With this norm, $B(G)$ is a unital Banach algebra.

Given $\phi\in B(G)$, a representation $(\pi,\mc H)$, and $x,y\in\mc H$ such that $\phi(g) = \la \pi(g)x,y\ra$, we have
$$
\|\phi\| \leq \|x\|\|y\|,
$$
and conversely, it is always possible to find $(\pi,\mc H)$ and vectors $x,y\in \mc H$ such that $\phi(g) = \la \pi(g)x,y\ra$ and $\|\phi\| = \|x\|\|y\|$.

The Fourier algebra of $G$ is denoted $A(G)$ and is the closure of the set of compactly supported functions in $B(G)$, and $A(G)$ is in fact an ideal. The Fourier algebra coincides with the set of all matrix coefficients of the left regular representation of $G$,
$$
A(G) = \{ \la \lambda(\cdot)x,y \ra \mid x,y\in L^2(G) \},
$$
and given any $\phi\in A(G)$, there are $x,y\in L^2(G)$ such that $\phi(g) = \la \lambda(g)x,y\ra$ and $\|\phi\| = \|x\|\|y\|$. This can be rephrased as follows. Given $\phi\in A(G)$, there are $f,h\in L^2(G)$ such that $\phi = f * \check h$ and $\|\phi\| = \|f\|\|h\|$, where $\check h(g) = h(g^{-1})$. This is often written as
$$
A(G) = L^2(G) * L^2(G).
$$
It is known that $\|\phi\|_\infty \leq \|\phi\|$ for any $\phi\in B(G)$, and hence $A(G) \subseteq C_0(G)$.

Although we will not study group von Neumann algebras in this paper, we note that $A(G)$ may be identified with the predual of the group von Neumann algebra $L(G)$ of $G$ via the duality
$$
\la T,\phi\ra = \la Tf,h\ra,
$$
where $T\in L(G)$ and $\phi = \bar h*\check f$ for some $f,h\in L^2(G)$.

\section{Proof of Theorem~\ref{thm:strategy}}\label{sec:strategy}

In this section we prove Theorem~\ref{thm:strategy}, which is the basis for proving Theorems~\ref{thm:fourier-PP} and \ref{thm:fourier-P}. We first prove that the conditions in Theorem~\ref{thm:strategy} ensure that the regular representation is completely reducible.

\begin{lem}
Let $G$ be a locally compact group. Any unitary representation of $G$ on a separable Hilbert space has at most countably many inequivalent (with respect to unitary equivalence) irreducible subrepresentations.
\end{lem}
\begin{proof}
Let $\pi$ be a unitary representation of $G$. The subrepresentations of $\pi$ are in correspondence with the projections in the commutant $\pi(G)'$, equivalent subrepresentations correspond to projections that are equivalent in $\pi(G)'$ (in the sense of Murray-von Neumann), and the irreducible subrepresentations correspond to minimal projections in $\pi(G)'$. It is therefore enough to show that a von Neumann algebra on a separable Hilbert space has at most countably many inequivalent minimal projections. Let $M$ be such a von Neumann algebra.

Recall that two minimal projections are inequivalent if and only if their central supports are orthogonal (see \cite[Proposition 6.1.8]{MR1468230}). Let $(p_i)_{i\in I}$ be a family of inequivalent minimal projections, and let $c_i$ be the central support of $p_i$. Then $(c_i)_{i\in I}$ is a family of orthogonal projections. By separability of the Hilbert space, $I$ must be countable. Hence there are at most countably many inequivalent minimal projections in $M$.
\end{proof}

The left regular representation represents $G$ on the Hilbert space $L^2(G)$, and if the group $G$ is second countable, the space $L^2(G)$ is separable. Hence we obtain the following corollary.

\begin{cor}
Let $G$ be a locally compact, second countable group. Then the left regular representation of $G$ has at most countably many inequivalent irreducible subrepresentations.
\end{cor}

We recall that a unitary representation is of type I, if the image of the representation generates a type I von Neumann algebra. A locally compact group is said to be of type I, if all its unitary representations are of type I (see \cite[Chapter~13]{MR0458185}). Disintegration theory works especially well in the setting of type I groups. We refer to \cite[Chapter~7]{MR1397028} for more on type I groups and disintegration theory. Several equivalent characterizations of type I groups can also be found in \cite[Chapter~9]{MR0458185}, but let us just mention one characterization here. The unitary equivalence classes of irreducible representations form a set $\hat G$ called the unitary dual of $G$. The unitary dual $\hat G$ is equipped with the Mackey Borel structure, and $G$ is of type I if and only if $\hat G$ is a standard Borel space.

\begin{prop}\label{prop:completely-reducible}
Let $G$ be a second countable, locally compact group satisfying the following two conditions.
\begin{enumerate}
	\item $G$ is of type I.
	\item There is a non-compact, closed subgroup $H$ of $G$ such that every irreducible unitary representation of $G$ is either trivial on $H$ or is a subrepresentation of the left regular representation $\lambda_G$.
\end{enumerate}
Then the left regular representation $\lambda_G$ is completely reducible.
\end{prop}
\begin{proof}
For each $p\in\hat G$, we let $\pi_p$ denote a representative of the class $p$, and we assume that the choice of representative is made in a measurable way (\cite[Lemma 7.39]{MR1397028}). We write the left regular representation as a direct integral of irreducibles,
$$
\lambda_G = \int_{\hat G}^\oplus n_p\pi_p \,d\mu(p),
$$
where $\mu$ is a Borel measure on $\hat G$ and $n_p\in\{0,1,2,\ldots,\infty\}$ (see \cite[Theorem 7.40]{MR1397028}). Let $A = \{p \in\hat G \mid \pi_p(h) = 1 \text{ for all } h\in H \}$ and let $B = \hat G \setminus A$. It is not hard to check that $A\subseteq \hat G$ is a Borel set for the Mackey Borel structure.

We note that if $\pi_p\in B$, then by assumption $\pi_p$ is a subrepresentation of $\lambda_G$. By the previous corollary, $B$ is countable. Since $\lambda_G$ has no subrepresentation which is trivial on a non-compact subgroup, we must have $\mu(A) = 0$. Then
$$
\lambda_G = \int_{B}^\oplus n_p\pi_p \,d\mu(p),
$$
and since $B$ is countable, $\lambda_G$ is a direct sum of irreducibles.
\end{proof}

When $\pi$ is a representation and $\alpha$ is a cardinal number we denote by $\pi^\alpha$ the direct sum of $\alpha$ copies of $\pi$. We say that $\pi^\alpha$ is a multiple of $\pi$.

\begin{lem}\label{lem:split}
Let $G$ be a locally compact, second countable group with left regular representation $\lambda_G$ and a closed subgroup $H$ such that
\begin{enumerate}
	\item $G$ is of type I;
	\item Every irreducible unitary representation of $G$ is either trivial on $H$ or is a subrepresentation of $\lambda_G$;
	\item $\lambda_G$ is completely reducible.
	\end{enumerate}
Then every unitary representation $\pi$ of $G$ is a sum $\sigma_1\oplus\sigma_2$, where $\sigma_1$ is trivial on $H$ and $\sigma_2$ is a subrepresentation of a multiple of $\lambda_G$.
\end{lem}
\begin{proof}
As in the previous proof, the basic idea is to use disintegration theory. However, this idea only applies if $\pi$ is a representation on a separable Hilbert space. There is a standard way of getting around the issue of separability: By an application of Zorn's lemma, we may write $\pi$ is a direct sum $\bigoplus_i \pi_i$ of cyclic representations $\pi_i$, so clearly it is enough to prove the lemma under the additional assumption that $\pi$ is cyclic. Since $G$ is second countable, $\pi$ then represents $G$ on a separable Hilbert space.

For each $p\in\hat G$, we let $\pi_p$ denote a representative of the class $p$, and we assume that the choice of representative is made in a measurable way (\cite[Lemma 7.39]{MR1397028}). 

We may write $\pi$ as a direct integral of irreducibles,
$$
\pi = \int_{\hat G}^\oplus n_p\pi_p \,d\mu(p),
$$
where $\mu$ is a Borel measure on $\hat G$ and $n_p\in\{0,1,2,\ldots,\infty\}$ (see \cite[Theorem 7.40 ]{MR1397028}). Let $A = \{p \in\hat G \mid \pi_p(h) = 1 \text{ for all } h\in H \}$ and let $B = \hat G \setminus A$. Then $A\subseteq \hat G$ is a Borel set. By assumption, there is a decomposition
$$
\lambda_G = \bigoplus_{p\in C} m_p \pi_p
$$
for some countable $C\subseteq \hat G$ and suitable multiplicities $m_p\in\{1,2,\ldots,\infty\}$. Also, it follows from our assumptions that $B\subseteq C$. If
$$
\sigma_1 = \int_{A}^\oplus n_p\pi_p \,d\mu(p), \qquad \sigma_2 = \int_{B}^\oplus n_p\pi_p \,d\mu(p),
$$
then we see that
$$
\pi = \sigma_1 \oplus \sigma_2,
$$
where $\sigma_1$ is trivial on $H$. As $B$ is countable, the integral defining $\sigma_2$ is actually a direct sum, so that $\sigma_2$ is a subrepresentation of 
$$
\bigoplus_{p\in B} n_p \pi_p
$$
which in turn is a subrepresentation of $\lambda_G\oplus\lambda_G\oplus\cdots$. Hence $\sigma_2$ is a subrepresentation of a multiple of $\lambda_G$.
\end{proof}

\begin{lem}\label{lem:AB}
Let $G$ be a locally compact group with left regular representation $\lambda_G$ and a closed, non-compact subgroup $H$. Suppose every unitary representation $\pi$ of $G$ is a sum $\sigma_1\oplus\sigma_2$, where $\sigma_1$ is trivial on $H$ and $\sigma_2$ is a subrepresentation of a multiple of $\lambda_G$. Then $A(G) = B(G)\cap C_0(G)$.
\end{lem}
\begin{proof}
The inclusion $A(G) \subseteq B(G) \cap C_0(G)$ holds for any locally compact group $G$. Suppose $\phi\in B(G)\cap C_0(G)$. Then there are a unitary representation $\pi$ of $G$ on some Hilbert space $\mc H$ and vectors $x,y\in\mc H$ such that
$$
\phi(g) = \la \pi(g)x,y\ra \qquad\text{for all } g\in G.
$$
By assumption we may split $\pi = \sigma_1\oplus\sigma_2$. Accordingly, we split $\phi = \phi_1 + \phi_2$, where $\phi_1$ is a coefficient of $\sigma_1$ etc. We will show that $\phi_1 = 0$ and $\phi_2\in A(G)$, which will complete the proof.

Since $\sigma_2$ is a subrepresentation of a multiple of $\lambda_G$, we see that $\phi_2$ is of the form
$$
\phi_2(g) = \sum_i \la \lambda_G(g)x_i,y_i\ra
$$
for some $x_i,y_i\in L^2(G)$ with $\sum_i \| x_i\|^2 < \infty$ and $\sum_i \| y_i\|^2 < \infty$. Each of the maps
$$
g\mapsto \la \lambda_G(g)x_i,y_i\ra
$$
is in $A(G)$ with norm at most $\|x_i\|\|y_i\|$. Since $A(G)$ is a Banach space and $\sum_i \|x_i\|\|y_i\| < \infty$, we deduce that $\phi_2 \in A(G)$, and in particular $\phi_2\in C_0(G)$. It then follows that $\phi_1\in C_0(G)$. Since $\sigma_1$ is trivial on $H$, we see that $\phi_1$ is constant on $H$ cosets. Since $H$ is non-compact, we deduce that $\phi_1 = 0$. Then $\phi = \phi_2\in A(G)$. This proves $B(G)\cap C_0(G) = A(G)$.
\end{proof}

Theorem~\ref{thm:strategy} is an easy consequence of the previous statements.

\begin{proof}[Proof of Theorem~\ref{thm:strategy}]
We assume that the locally compact, second countable group $G$ satisfies the two conditions in the statement of the theorem. It follows from Proposition~\ref{prop:completely-reducible} that $\lambda_G$ is completely reducible. By Lemma~\ref{lem:split}, every unitary representation $\pi$ of $G$ is a sum $\sigma_1\oplus\sigma_2$, where $\sigma_1$ is trivial on $H$ and $\sigma_2$ is a subrepresentation of a multiple of $\lambda_G$. From Lemma~\ref{lem:AB} we conclude that $A(G) = B(G) \cap C_0(G)$.
\end{proof}

\section{Invariant measures on homogeneous spaces}\label{sec:inv-meas}

To describe the irreducible representations of the groups $P$ in Theorems~\ref{thm:fourier-PP} and~\ref{thm:fourier-P}, we rely on a general method known to the common man as the Mackey Machine. Essential in the Mackey Machine is the notion of induced representations. For a general introduction to the theory of induced representations we refer to \cite[Chapter 6]{MR1397028} which also contains a description of (a simple version of) the Mackey Machine. The general results about the Mackey Machine can be found in the original paper \cite{MR0098328}. See also the book \cite{MR3012851}.

The construction of an induced representation from a closed subgroup $H$ to a group $G$ is more easily described when the homogeneous space $G/H$ admits an invariant measure for the $G$-action given by left translation. Regarding homogeneous spaces and invariant measures we record the following easy (and well-known) facts.

\begin{lem}\label{lem:cosetspaces}
Consider topological groups $G$, $N$, $H$, $K$, $A$, $B$ and topological spaces $X$ and $Y$.

\begin{enumerate}
	\item 
Suppose $G$ is the semi-direct product $G = N\rtimes H$, where $N$ is normal in $G$. If $K\leq H$ is a closed subgroup of $H$, then there is a canonical isomorphism
$$
NH/NK \simeq H/K
$$
as $G$-spaces. Here the $G$-action on $H/K$ is the $H$-action, and $N$ acts trivially on $H/K$.

\item 
Suppose $G = N\times H$, and $A\leq N$, $B\leq H$ are closed subgroups. Then there is a canonical isomorphism
$$
(N\times H)/(A\times B) \simeq N/A \times H/B
$$
as $G$-spaces, where the $G$-action on $N/A \times H/B$ is the product action of $N\times H$.

\item
Suppose $G\acts X$ and $H\acts Y$ have invariant, $\sigma$-finite Borel measures. Then the product $G\times H \acts X\times Y$ has an invariant, $\sigma$-finite Borel measure.

\item
Suppose $G$ is compact (or just locally compact, amenable) and $X$ is compact. Then any action $G\acts X$ has an invariant probability measure.
\end{enumerate}
\end{lem}
\begin{proof}
\mbox{}
\begin{enumerate}
	\item The map $[nh]_{NK} \mapsto [h]_K$ is a well-defined, equivariant homeomorphism.
	\item The map $[(n,h)]_{A\times B} \mapsto ([n]_A , [h]_B)$ is a well-defined, equivariant homeomorphism.
	\item Take the product measure on $X\times Y$ of the invariant measures on $X$ and $Y$.
	\item This is Proposition~5.4 in \cite{MR767264}.\qedhere
\end{enumerate}
\end{proof}

The following lemma is often useful when one wants to verify condition (2) of Theorem~\ref{thm:strategy}.

\begin{lem}\label{lem:trivial-induction}
Let $G$ be a locally compact group with closed subgroups $N\subseteq H\subseteq G$, and suppose $N\triangleleft G$. If $\sigma$ is a unitary representation of $H$ which is trivial on $N$, and if $G/H$ admits a $G$-invariant measure, then the induced representation $\Ind_H^G \sigma$ is also trivial on $N$.
\end{lem}
\begin{proof}
The kernel of an induced representation can even be described explicitly (see e.g. \cite[Theorem~2.45]{MR3012851}).
\end{proof}

\section{Proof of Theorem~\ref{thm:fourier-PP}}\label{sec:PP}

In this section we prove Theorem~\ref{thm:fourier-PP}. Let $P$ be the group defined in \eqref{eq:PP}. We note first that the group $P$ is of type I: the group $P$ is the connected component of a real algebraic group, and such groups are of type I according to \cite[Theorem 1]{MR0099380}.

The unitary dual of $P$, i.e. the equivalence classes of the irreducible representations of $P$, can be determined using the Mackey Machine (see e.g. \cite[Chapter~6]{MR1397028} and \cite{MR3012851}). Observe that $P = N_0\rtimes P_0$, where
\begin{align}
\label{eq:Pzero}
P_0 &= \left\{\begin{pmatrix}
\lambda & a & 0 \\
0 & \lambda^{-1} & 0 \\
0 & 0 & 1
\end{pmatrix}
\middle\vert\ a\in\R,\ \lambda>0
\right\},
\displaybreak[0] \\[15pt]
N_0 &= \left\{\begin{pmatrix}
1 & 0 & c \\
0 & 1 & b \\
0 & 0 & 1
\end{pmatrix}
\middle\vert\ b,c\in\R
\right\}.
\end{align}
As in \eqref{eq:SL3}, we identify $N_0\simeq\R^2$ and consider $P_0\subseteq\SL_2(\R)$ in the obvious way so that $P = N_0\rtimes P_0$ is a subgroup of $\R^2\rtimes\SL_2(\R)$, where $\SL_2(\R)$ acts on $\R^2$ by matrix multiplication. The dual action $P_0\acts\hat N_0$ is then given by $(p.\nu)(n) = \nu(p^{-1}.n)$ for $p\in P_0$, $\nu\in\hat N_0$, and $n\in N_0$. Under the usual identification $\hat N_0 \simeq \R^2$ we see that $p\in P_0$ acts on $\R^2$ by matrix multiplication by the transpose of the inverse of $p$. Thus, if $p$ has the form in \eqref{eq:Pzero}, then the action of $p$ on $\R^2$ is
$$
\begin{pmatrix}
	s \\ t
\end{pmatrix}
\mapsto
\begin{pmatrix}
	\lambda^{-1} & 0 \\
	-a & \lambda
\end{pmatrix}
\begin{pmatrix}
	s \\ t
\end{pmatrix}.
$$
There are five orbits in $\hat N_0$ under this action, which give rise to five families of irreducible representations of $P$. As representatives of the orbits we choose the points
$$
\binom10,\
\binom{-1}0,\
\binom01,\
\binom0{-1},\
\binom00.
$$
The first two points $(1,0)$ and $(-1,0)$ have trivial stabilizers in $P_0$ and give rise to two irreducible representations $\pi^+$ and $\pi^-$. The union of the orbits of $(\pm 1,0)$ has complement of (Plancherel) measure zero, and the regular representation $\lambda_P$ of $P$ is the countably infinite direct sum of $\pi^+\oplus\pi^-$ (see e.g. \cite[Section~1]{MR509261} for a proof).

Let $\nu\in\hat N_0$ be represented by one of the points $(0,1)$ or $(0,-1)$. Then $\nu$ is trivial on the subgroup $N_1\subseteq N_0$ defined as
\begin{align}
\label{eq:N1}
N_1 &= \left\{\begin{pmatrix}
1 & 0 & c \\
0 & 1 & 0 \\
0 & 0 & 1
\end{pmatrix}
\middle\vert\ c\in\R
\right\}.
\end{align}
 The stabilizer of $\nu$ in $P_0$ is the subgroup
\begin{align}
\label{eq:P1}
P_1 &= \left\{\begin{pmatrix}
1 & a & 0 \\
0 & 1 & 0 \\
0 & 0 & 1
\end{pmatrix}
\middle\vert\ a\in\R
\right\},
\end{align}
and $\nu$ extends to a character $\tilde\nu$ of $N_0P_1$ being trivial on $P_1$. An irreducible representation of $P$ obtained from $\nu$ is of the form
$$
\pi = \Ind_{N_0P_1}^P (\tilde\nu\otimes\gamma),
$$
where $\gamma\in\hat P_1$. The representation $\tilde\nu\otimes\gamma$ is clearly trivial on $N_1$. Since $N_0P_1$ is normal in $P$, the homogeneous $P/(N_0P_1)$ has a $P$-invariant measure. It follows from Lemma~\ref{lem:trivial-induction} that $\pi$ is trivial on $N_1$.

The last point $(0,0)$ has stabilizer $P_0$ and determines the irreducible representations of $P$ that factor through $P_0$, i.e. are trivial on $N_0$. Such representations are clearly trivial on $N_1$.

Since the group $N_1$ is non-compact, we have verified the conditions of Theorem~\ref{thm:strategy} for the group $P$. This proves Theorem~\ref{thm:fourier-PP}.

\section{Minimal parabolic subgroups}\label{sec:simple-groups}

In this section we prove Theorem~\ref{thm:fourier-P}. Let $\mb F$ be one of the four division algebras: the real numbers $\R$, the complex numbers $\C$, the quaternions $\mb H$, or the octonions $\mb O$. Let $\mb F^{p,q}$ denote the real vector space $\mb F^{p+q}$ equipped with the hermitian form
$$
\la x,y\ra = \sum_{i=1}^p  x_i\bar y_i - \sum_{i=p+1}^{p+q}  x_i\bar y_i.
$$
We also think of $\mb F^{p,q}$ as a right $\mb F$-module. Of course, $\mb F^n = \mb F^{n,0}$. Let $G$ be one of the following groups:
\begin{align*}
\SO_0(n,1) &= \text{the identity component of the orthogonal group of } \R^{n,1}; \\
\SU(n,1) &= \text{the special unitary group of } \C^{n,1}; \\
\Sp(n,1) &= \text{the symplectic (quaternion--unitary) group of } \mb H^{n,1}; \\
\FF      &= \text{the exceptional rank one group of type } F_4.
\end{align*}

Any connected simple Lie group of real rank one is locally isomorphic to one of the groups above (see e.g. \cite[p.~426]{MR1920389}). A thorough account of the exceptional group $\FF$ can be found in \cite{MR560851}.

Let $G = KAN$ be an Iwasawa decomposition of $G$. Then $K$ is a maximal compact subgroup, $A$ is abelian of dimension $1$, and $N$ is nilpotent. Let $M$ be the centralizer of $A$ in $K$, and let $P = MAN$ be the minimal parabolic subgroup of $G$. Let $Z$ denote the center of the nilpotent group $N$. Then $Z$ is a non-compact subgroup of $P$.

The unitary dual of $P$ can be determined using the Mackey Machine, and we describe it below. Details can be found in \cite{MR0422519} on which this is based (see also \cite{MR0507242}). With the knowledge of the unitary dual of $P$, it is not difficult to verify the conditions of Theorem~\ref{thm:strategy} for the group $P$. We do this in Proposition~\ref{prop:dichotomy} and Proposition~\ref{prop:typeI} below, and this then completes the proof of Theorem~\ref{thm:fourier-P}.

Let $\chi$ be a character of $N$, and let $M_\chi$ be the stabilizer subgroup in $MA$ of $\chi$ (the action of $MA$ on $\hat N$ is the dual action of the conjugation action of $MA$ on $N$). If $\chi$ is the trivial character, then $M_\chi$ is $MA$, but otherwise it is not difficult to show that $M_\chi$ is a closed subgroup of $M$.

The character $\chi$ extends to a character of $M_\chi N$ being trivial on $M_\chi$. With $\gamma\in \hat M_\chi$, we obtain an irreducible unitary representation $\pi_{\chi,\gamma}$ of $P$ by induction,
$$
\pi_{\chi,\gamma} = \Ind_{M_\chi N}^{MAN} (\chi\otimes\gamma).
$$
The remaining irreducible unitary representations of $P$ do not arise from characters on $N$. These occur only when $\mb F\neq\R$. Let $\lambda$ be a non-zero functional on $\mathfrak z = \ImF$, the Lie algebra of $Z$. It is known that there exists an infinite dimensional irreducible representation $\eta_\lambda$ of $N$, uniquely determined be the property
$$
\eta_\lambda(zn) = e^{i\lambda(\log z)}\eta_\lambda(n), \qquad z\in Z,\ n\in N.
$$
Moreover, $\eta_\lambda$ is uniquely determined within unitary equivalence by the central character $\lambda$ (see \cite[Lemma~4.4]{MR0422519}).
Let $M_\lambda$ denote the stabilizer in $MA$ of $\eta_\lambda$. Then $\eta_\lambda$ extends to a representation of $M_\lambda N$ as discussed in \cite[Sections~7 and 8]{MR0422519}.

With $\gamma\in\hat M_\lambda$, we obtain an irreducible unitary representation $\pi_{\lambda,\gamma}$ of $P$ by induction,
$$
\pi_{\lambda,\gamma} = \Ind_{M_\lambda N}^{MAN} (\eta_\lambda\otimes\gamma).
$$
This completes the description of the unitary dual of $P$.

\begin{prop}\label{prop:dichotomy}
Any irreducible unitary representation of $P$ is either trivial on the center $Z$ of $N$ or is a subrepresentation of $\lambda_P$.
\end{prop}
\begin{proof}
Consider first a representation $\pi_{\chi,\gamma}$, where $\chi\in\hat N$ is a character. If $\chi$ is the trivial character, then $\pi_{\chi,\gamma}$ factors through $P/N$ and is clearly trivial on $Z$.

Suppose next that $\chi$ is a non-trivial character that annihilates $Z$. Since $Z$ is normal in $P$, it follows from Lemma~\ref{lem:trivial-induction} that $\pi_{\chi,\gamma}$ is trivial on $Z$, once we show that the homogeneous space $P / M_\chi N$ admits a $P$-invariant measure. Using Lemma~\ref{lem:cosetspaces} we find
$$
P / M_\chi N \simeq MA / M_\chi \simeq M/M_\chi \times A.
$$
The left translation action $A\acts A$ has the Haar measure as an invariant measure. Since $M$ is compact, the action $M\acts M/M_\chi$ has an invariant measure. It follows that $P\acts M_\chi N$ has an invariant measure, and then by Lemma~\ref{lem:trivial-induction} the representation $\pi_{\chi,\gamma}$ is trivial on $Z$.

Suppose instead that $\chi$ does not annihilate $Z$. This happens only when $\mb F = \R$. The stabilizer group $M_\chi$ is compact, and hence $\gamma\in\hat M_\chi$ is a subrepresentation of the regular representation of $M_\chi$. If $\dim N\neq 1$, then the action of $MA$ on the non-zero characters of $N$ is transitive. In particular, the orbit has positive Plancherel measure in $\hat N$. If $\dim N = 1$, then the orbit of $\chi$ is either $\R_+$ or $\R_-$ inside $\hat N \simeq \R$, and both of these sets have positive measure. By \cite[Corollary~11.1]{MR0342641} we conclude that $\pi_{\chi,\gamma}$ is a subrepresentation of $\lambda_P$.

Consider finally a representation $\pi_{\lambda,\gamma}$, where $\lambda$ is a non-zero functional on the Lie algebra $\mathfrak z$ of $Z$. In this case, $\mb F\neq\R$.

It is not difficult to show that $M_\lambda$ is a closed subgroup of $M$ and hence compact. Therefore, $\gamma\in\hat M_\lambda$ is a subrepresentation of the regular representation of $M_\lambda$. Again by \cite[Corollary~11.1]{MR0342641}, to conclude that $\pi_{\lambda,\gamma}$ is a subrepresentation of $\lambda_P$, it remains to show that the orbit of $\eta_\lambda$ in $\hat N$ has positive Plancherel measure.

The characters in $\hat N$ all annihilate the non-compact group $Z$ and hence must form a null set for the Plancherel measure. If $\mb F = \mb H$ or $\mb F = \mb O$, then $MA$ acts transitively on $\{\eta_\lambda \in \hat N \mid \lambda\in\mathfrak z^*\setminus \{0\}\}$, and therefore the orbit of $\eta_\lambda$ must have positive Plancherel measure.

If $\mb F = \C$, then the action of $MA$ on the representations $\{\eta_\lambda \in \hat N \mid \lambda\in\mathfrak z^*\setminus \{0\}\}$ has two orbits. The group $N$ is the Heisenberg group of dimension $2n -1$, and the Plancherel measure for the Heisenberg group can be found in \cite[p.~241]{MR1397028}. We see that the measure of the orbit of $\eta_\lambda$ is
$$
\mu_N(MA.\eta_\lambda) = \int_0^\infty |h|^{n-1} \,dh.
$$
Hence the orbit of $\eta_\lambda$ has positive, in fact infinite, measure. By \cite[Corollary~11.1]{MR0342641}, we conclude that $\pi_{\lambda,\gamma}$ is a subrepresentation of $\lambda_P$.
\end{proof}

\begin{prop}\label{prop:typeI}
The minimal parabolic subgroup $P$ is of type I.
\end{prop}
\begin{proof}
We apply \cite[Theorem~9.3]{MR0098328} to show that $P$ is of type I. It is known that connected nilpotent Lie groups are of type I (see \cite[Corollaire~4]{MR0115097}), and it follows that $N$ is of type I. Hence $\hat N$ is a standard Borel space. One can check that the action $MA\acts \hat N$ has only finitely many orbits (the exact number depends on $\mb F$ and $n$), so in particular there is a Borel set in $\hat N$ which meets each orbit exactly once. By \cite[Theorem~9.2]{MR0098328} the action $MA\acts \hat N$ is \emph{regular}, that is, $N$ is \emph{regularly embedded} in $P$.

We now verify that when $\pi\in \hat N$, the stabilizer $L_\pi = \{ g\in MA \mid g.\pi \simeq \pi\}$ is of type I. Indeed, if $\pi$ is the trivial character on $N$, then $L_\pi = MA$ is a direct product of the compact group $M$ and the abelian group $A$. Hence the stabilizer $MA$ is of type I. If $\pi$ is not the trivial character, then $L_\pi$ is a closed subgroup of $M$, hence compact. In particular the stabilizers are of type I. According to \cite[Theorem~9.3]{MR0098328} we may now conclude that $P$ is of type I.
\end{proof}

\section{Concluding remarks}\label{sec:concluding}
Theorem~\ref{thm:fourier-P} shows that Question~\ref{q1} has a positive answer for the minimal parabolic subgroups $P = MAN$ in the groups $\SO_0(n,1)$, $\SU(n,1)$, $\Sp(n,1)$ and $\FF$. One could ask if the same is true for the smaller groups $MN$, $AN$ or $N$. We will now discuss these cases. Recall from the introduction that non-compact second countable connected unimodular groups never satisfy \eqref{eq:q1}.

Let $G$ be one of the classical groups $\SO_0(n,1)$, $\SU(n,1)$, $\Sp(n,1)$, with $n\geq 2$, or the exceptional group $\FF$. Let $\mb F$ be the corresponding division algebra, $\R$, $\C$, $\mb H$, or $\mb O$. We start by discussing the groups $N$. Since $N$ is nilpotent, $N$ is unimodular. Indeed, a locally compact group $G$ is unimodular if and only if $G/Z$ is unimodular, where $Z$ is the center of $G$ (see \cite[p.~92]{MR0175995}). Induction on the length of an upper central series then shows that all locally compact nilpotent groups are unimodular. Since $N$ is also connected, it follows that
$$
A(N) \neq B(N) \cap C_0(N).
$$

Next we discuss the groups $MN$. Since $MN$ is a semi-direct product of the unimodular group $N$ by the compact group $M$, we will argue that $MN$ itself is unimodular. Indeed, this follows directly from \cite[Proposition~23]{MR0175995} but we also include another argument here. If we use $\Delta_G$ to denote the modular function of a locally compact group $G$, then since $N$ is normal in $MN$, we have $\Delta_{MN}|N = \Delta_N = 1$. Also, since $M$ is compact, $\Delta_{MN}|M = 1$. Since $M$ and $N$ generate $MN$, it follows that $\Delta_{MN} = 1$. So $MN$ is connected and unimodular, and hence
$$
A(MN) \neq B(MN) \cap C_0(MN).
$$

Alternatively, one could show that all orbits in $\hat N$ under the action of $M$ have zero Plancherel measure. This type of argument will be used below for the groups $AN$.

For the groups $\SO_0(n,1)$, $\Sp(n,1)$, and $\FF$ it will usually also be the case that Question~\ref{q1} has a negative answer for the groups $AN$ as well. However, there is one exception. If $G = \SO_0(2,1)$, then $M$ is trivial and $P$ coincides with $AN$. Hence it follows from Theorem~\ref{thm:fourier-P} that Question~\ref{q1} has an affirmative answer for the group $AN$. In this special case let us remark that $AN$ is in fact isomorphic to the $ax + b$ group, and the result that $A(AN) = B(AN)\cap C_0(AN)$ is actually the original result of Khalil from \cite{MR0350330}.

The unimodularity argument used for the groups $N$ and $MN$ does not apply to $AN$, since these groups are not unimodular (see \cite[(1.14)]{MR0507242}). As mentioned in the introduction, a group satisfying \eqref{eq:q1} has a completely reducible left regular representation, and in particular the left regular representation has irreducible subrepresentations. Then by \cite[Corollary~11.1]{MR0342641} at least one of the orbits of the action $A\acts\hat N$ must have positive Plancherel measure. To show that $A(AN) \neq B(AN) \cap C_0(AN)$ it therefore suffices to show that any orbit of $A\acts\hat N$ has zero Plancherel measure.

At this point we split the argument in cases. Consider first the case when $\mb F = \R$ and $n\geq 3$. Then $N \simeq \R^{n-1}$, and the Plancherel measure on $\hat N \simeq \R^{n-1}$ is the Lebesgue measure. Since $A$ acts on $\hat N$ by dilation, every orbit except $\{0\}$ is a half-line. Since $n\geq 3$, every half-line in $\R^{n-1}$ has vanishing Lebesgue measure, and hence every orbit in $\hat N$ has vanishing Plancherel measure.

Consider now the other cases where $\mb F$ is $\C$, $\mb H$, or $\mb O$. As mentioned earlier, the unitary dual $\hat N$ then consists of characters and the infinite dimensional representations $\hat N_r = \{\eta_\lambda\}_\lambda$ (see Section~\ref{sec:simple-groups}).

Fortunately, the Plancherel measure for $N$ is known. It is described in \cite[Section~3]{MR996553}. Since the characters are trivial on the center $Z$ of $N$ which is non-compact, the characters form a null set for the Plancherel measure. Let $k$ be the dimension of $Z$ so that $k$ is either 1, 3 or 7. If we identify $\hat N_r$ with $\mathfrak z^*\setminus\{0\}$ (the non-zero functionals on the Lie algebra of $Z$) which in turn is identified with the punctured Euclidean space $\R^k\setminus\{0\}$, then it follows from \cite[p.~524]{MR996553} that the Plancherel measure on $\hat N_r$ is absolutely continuous with respect to the Lebesgue measure.

Since $A$ acts on $\hat N_r$ by dilation, every orbit in $\hat N_r$ is a half-line. Every half-line has vanishing Lebesgue measure, unless $k=1$, and hence every orbit in $\hat N_r$ has vanishing Plancherel measure, except when $\mb F = \C$. Combined with the fact that the characters have vanishing Plancherel measure, we conclude that every orbit in $\hat N$ has vanishing Plancherel measure. We collect the discussion above in the following proposition.

\begin{prop}
Let $G$ be one of the simple Lie groups $\SO_0(n,1)$ ($n\geq 3$), $\Sp(n,1)$ ($n\geq 2$), or $\FF$. Let $MAN$ be the minimal parabolic subgroup of $G$. If $H$ is either $N$, $MN$, or $AN$ then
$$
A(H) \neq B(H)\cap C_0(H).
$$
\end{prop}

As pointed out, the argument breaks down when $\mb F = \C$. So finally, we consider the group $AN$ in $G = \SU(n,1)$.

\begin{prop}
Let $G$ be the simple Lie group $\SU(n,1)$ ($n\geq 2$) with Iwasawa decomposition $G = KAN$. Then 
$$
A(AN) = B(AN)\cap C_0(AN).
$$
\end{prop}
\begin{proof}
We will verify the conditions of Theorem~\ref{thm:strategy} for the group $AN$.

First we verify that $AN$ is a group of type I. We mimic the proof of Proposition~\ref{prop:typeI}. The group $N$ is the Heisenberg group of dimension $2n-1$, and its unitary dual $\hat N$ consists of characters annihilating the center and the infinite dimensional representations $\{\eta_\lambda\}_\lambda$. Recall that $N$ is of type I, and hence $\hat N$ is a standard Borel space. 

The characters in $\hat N$, which we think of simply as $\C^{n-1}$, form an invariant subset whose orbits consist of the origin $\{0\}$ and half-lines originating at the origin. The infinite dimensional representations in $\hat N$, which we think of simply as $\R\setminus\{0\}$ also form an invariant subset which has two orbits, $\R_+$ and $\R_-$.

If $S$ denotes the unit sphere in $\C^{n-1}$, then $R = \{0\}\cup S \cup \{1,-1\}$ is a set of representatives for the orbits of $A\acts \hat N$. We claim that $R$ is a Borel subset of $\hat N$. To see this, it suffices to prove that $S$ is a Borel subset, since points are always Borel subsets in a standard Borel space.

The Fell topology on $\hat N$ is well-known (see e.g. \cite[Chapter~7]{MR1397028}). The characters form a closed subset in $\hat N$, and on the set of characters the Fell topology coincides with the Euclidean topology (on $\C^{n-1}$). In particular $S$ is closed in the Fell topology. By \cite[Theorem~7.6]{MR1397028}, the Mackey Borel structure on $\hat N$ is induced by the Fell topology, since $N$ is of type I. It follows that $S$ is a Borel set.

We conclude from \cite[Theorem~9.2]{MR0098328} that the action $A\acts \hat N$ is \emph{regular}, that is, $N$ is \emph{regularly embedded} in $AN$.

Next we verify that if $\pi\in \hat N$, then the stabilizer $A_\pi = \{ \alpha\in A \mid \alpha.\pi \simeq \pi\}$ is of type I. Indeed, if $\pi$ is the trivial character on $N$, then $A_\pi = A$ which is an abelian group. Hence the stabilizer $A$ is of type I. If $\pi$ is not the trivial character, then the stabilizer $A_\pi$ is trivial. So all stabilizers are of type I. According to \cite[Theorem~9.3]{MR0098328} we may now conclude that $AN$ is of type I.

The unitary dual of $AN$ is described in \cite[Proposition~7.6]{MR0422519}. The irreducible representations of $AN$ fall into three families: representations obtained from the trivial character of $N$, representations obtained from non-trivial characters of $N$, and representations obtained from infinite dimensional representations of $N$.

Representations obtained from the trivial character of $N$ annihilate $N$ and factor through $AN/N = A$. If $\pi = \Ind_N^{AN}\chi$ is a representation of $AN$ induced from a non-trivial character of $N$, then $\pi$ annihilates the center $Z$ of $N$ (Lemma~\ref{lem:trivial-induction}).

Finally, consider a representation $\pi = \Ind_N^{AN} \eta_\lambda$ where $\eta_\lambda\in\hat N$ is an infinite dimensional irreducible representation. As mentioned before, the action of $A$ on the representations $\{\eta_\lambda\}_\lambda$ has two orbits, $\R_+$ and $\R_-$. The Plancherel measure of these two orbits can be shown to be positive (see the last part of the proof of Proposition~\ref{prop:dichotomy}). By \cite[Corollary~11.1]{MR0342641} we conclude that $\pi$ is a subrepresentation of the left regular representation of $AN$.
\end{proof}


\end{document}